\documentclass{article}

\usepackage{graphicx,amsmath}
\usepackage{latexsym}
\usepackage{amssymb}

\newtheorem{assumption}{Assumption}
\newtheorem{definition}{Definition}
\newtheorem{example}{Example}
\newtheorem{lemma}{Lemma}
\newtheorem{theorem}{Theorem}
\newtheorem{corollary}{Corollary}

\newenvironment{proof}{\textbf{Proof}. \medskip}{\hfill \rule{7pt}{7pt}\\[10pt]}

\def\R#1{$(\ref{#1})$}
\def\RR{\mathbb{R}}
\def\E{\mathrm{e}}
\def\D{\,\mathrm{d}}

\newenvironment{meqn}
{\arraycolsep=1.4pt
  
  \begin{array}{rcl}}
  {\end{array}}

\begin{document}

\title{Explicit Volume-Preserving Splitting Methods for Polynomial Divergence-Free Vector Fields}

\author{
{\sc
Huiyan Xue
 and 
 Antonella Zanna\thanks{Email: Huiyan.Xue@math.uib.no, Antonella.Zanna@math.uib.no}}}

\maketitle

\begin{abstract}
{
In this paper, we present new, explicit, volume-preserving vector fields for polynomial divergence-free vector fields of arbitrary degree (both positive and negative). The main idea is to decompose the divergence polynomial by means of an appropriate basis for polynomials: the monomial basis. For each monomial basis function, the split fields are then identified by collecting the appropriate terms in the vector field so that each split vector field is volume preserving. We show that each split field can be integrated exactly by analytical methods. Thus, the composition yields a volume preserving numerical method. Our numerical tests indicate that the methods compare favorably to standard integrators both in the quality of the numerical solution and the computational effort.
}

\end{abstract}

\section{Introduction}
Divergence-free vector fields occur naturally in incompressible fluid dynamics, and preservation of phase-space volume is also a crucial ingredient in many, if not all, ergodic theorems. Preservation of volume by a numerical method for differential equations is thus a desirable property in the study of dynamical system. Nevertheless, designing volume preserving numerical integrators is a hard task, as the space of divergence-free vector fields seems to be too large. \cite{qin93vps}  proved that no standard numerical method, for example a  Runge--Kutta scheme, is volume preserving for all such vector fields. Thus, the design of efficient methods that preserve volume is still a standing open problem in geometric integration \cite{Mclachlan98asurvey}.

Despite no-go theorems for volume preservation within the class of ``standard'' methods (see also the recent results in \cite{MR2334044}, \cite{chartier2007preserving}, where it is proved that no B-series method can be volume preserving for all possible divergence-free vector fields), it is known that  volume preservation can be achieved, either restricting the class of vector fields or using methods other than B-series. \cite{chartier2007preserving} have shown that volume-preserving B-series can be obtained if the vector field has some specific dependence (for example $f_1=f_1(x_2), f_2 = f_2(x_3), \ldots, f_n = f_n(x_1)$, or if the variables naturally decompose into two sets, $\mathbf{x} = [\mathbf{y}, \mathbf{z}]^T $, obeying $\dot{\mathbf{y}} = \mathbf{g}(\mathbf{z}), \dot{\mathbf{z}} = \mathbf{h}(\mathbf{y})$). These two examples of vector fields correspond to what we call \emph{off-diagonal}. These are generally easier to treat and other methods will be described in the sequel.

Volume-preserving maps can be constructed techniques other than methods that possess B-series, for instance
using generating functions \cite{MR1370428, MR1290253,scovel91sni}. This technique involves evaluating definite integrals of the vector field, and, in addition, the method does not preserve fixed points. 
Another technique is based on splitting methods. One of the earliest volume-preserving splitting methods is indeed the splitting method by \cite{kang95vpa}, decomposing the vector field into the sum of essentially 2-dimensional Hamiltonian fields, which are then solved by a (typically implicit) symplectic method. 

Because of the difficulty of addressing the general space of divergence-free vector fields, recent efforts have concentrated to smaller, yet still interesting, functions spaces, for instance the space of polynomial fields. An earlier paper on splitting polynomial vector fields is by \cite{mclachlan04egi}. That paper had some discussion of  the divergence-free case, but mainly dealt with the Hamiltonian case. 
Investigations of the Hamiltonian case, which involves expressing a scalar polynomial of degree $d$ in $n$ variables as a sum of functions of fewer variables, have shown that good splitting methods exist, but that finding and analyzing them (especially for general $n$ and $d$) is very difficult \cite{dragt96sma,mclachlan04egi,rangarajan96sco}. The volume-preserving case, which involves $n$ polynomials subject to the divergence-free condition, is even harder, although there is a conjecture by \cite{mclachlan04egi} that they can be expressed as a sum of $n+d$ shears, each a function of $n-1$ variables. The case of linear and quadratic divergence-free vector fields was studied in detail in \cite{mclachlan09evp}, where several explicit volume-preserving splitting methods were introduced. 
In that paper, two  main classes of methods were considered: a) methods that distinguish the diagonal and off-diagonal part and b) methods that do not.
By diagonal part we mean all the terms of the vector field such that $\dot x_i$ depends on $x_i$, for $i=1, \ldots, n$. Similarly, the off-diagonal part refers to all the terms of the vector field such that $\dot x_i$ does not depend on $x_i$, $i=1,\ldots, n$.
Furthermore, for the class a), several explicit schemes that treated the diagonal part and the off-diagonal part separately were introduced and tested numerically. Numerical tests indicated that methods that treated the diagonal part by splitting it  in terms treated by exponentials had a smaller error than methods splitting in shears.

In this paper we present a new approach that allows us to develop explicit volume-preserving methods for arbitrary polynomial divergence-free vector fields, including those with negative degree. 
The first main insight is to expand the divergence equation, rather than the vector field, in the monomial basis. For each monomial basis element, we identify the elements in the vector field associated to it to construct a divergence-free elementary vector field. The second insight is to recognize that the elementary divergence-free vector fields  can all be treated by the same formalism and therefore can be solved explicitly by elementary analytical methods. The split fields are then composed to obtain explicit first order method and second order method (by symmetrization).  
The resulting composition method is thus explicit and volume-preserving. 
Being explicit, the proposed method are computationally efficient and possess excellent qualitative properties.
We believe that thisis a consequence of volume preservation solely, as the methods are not necessarily time-reversible nor self-adjoint (for instance, the first order method is neither).

The paper is organized as follows. In Section~\ref{sec:background} we review some background and introduce notation. In Section~\ref{sec:hom} we present the monomial basis for polynomial volume-preserving vector fields and prove that the vector fields associated to the basis elements can be integrated exactly. The case of polynomials of negative degree is treated in Section~\ref{sec:laurent}. In  Section~\ref{sec:num} we test the application to volume-preserving cubic Stokes flows. Finally, Section~\ref{sec:conclusion} is devoted to some concluding remarks and some directions for future research.

\section{Background and notation}
\label{sec:background}
We consider the ordinary differential equation
\begin{equation}
	\dot{\mathbf{x}} = \mathbf{f}(\mathbf{x}), \qquad \mathbf{x}(0) = \mathbf{x}_0,
	\label{eq:vf}
\end{equation}
where $\mathbf{x} \in \RR^n$ and $\mathbf{f} :\RR^n \to \RR^n$, $\mathbf{f}(\mathbf{x}) = [f_1(\mathbf{x}), \ldots, f_n(\mathbf{x})]^T$, is subject to the divergence-free condition
\begin{equation}
	\nabla \cdot \mathbf{f} = \sum_{i=1}^n \partial_{x_i} f_i (\mathbf{x})=0.
	\label{eq:df}
\end{equation}
An arbitrary vector field $\mathbf{f(x)}$ can always be decomposed as 
\begin{displaymath}
	\mathbf{f(x)}= \mathbf{f}^{\mathrm{diag}} (\mathbf{x}) +\mathbf{ f}^{\mathrm{offdiag}} (\mathbf{x}) 
\end{displaymath}
where, component-wise,  $f_i^{\mathrm{diag}} (\mathbf{x})$  is the collection of terms in $f_i(\mathbf{x})$ that depend on $x_i$ ( i.e.\ $\partial_{x_i}  f_i^{\mathrm{diag}} (\mathbf{x})\not=0$). Similarly, $\mathbf{ f}^{\mathrm{offdiag}} (\mathbf{x}) $, is given, component-wise, by the collection of terms in $f_i$ that do not depend on $x_i$, i.e.\ $ \partial_{x_i} f_i^{\mathrm{offdiag}} (\mathbf{x})=0$.
We refer to $\mathbf{f}^\mathrm{diag}$ and $\mathbf{ f}^{\mathrm{offdiag}} (\mathbf{x}) $ as the  \emph{diagonal part} and the \emph{off-diagonal} part  of the vector field $\mathbf{f}$, respectively.

From the definition of divergence, it is clear that only the coefficients of the diagonal part are directly involved in the divergence-free condition \R{eq:df}, therefore, vector fields with zero diagonal part are automatically divergence-free.

The off-diagonal part of a vector field is generally easier to treat by volume-preserving methods. The method of splitting in canonical  $n$-shears is generic (for each $i =1\ldots n$, we solve for $\dot{x}_i$, while $\dot{x}_k=0$, $k\not=i$). For polynomial systems, it is possible to construct splittings in lower-triangular systems (for each $i =1\ldots n$, we solve for $\dot{x}_i$ depending only on $x_1, \ldots, x_{i-1}$, under a suitable permutation of the indices), see \cite{mclachlan09evp}. 

Thus, to obtain volume-preserving methods, special care has to be taken when splitting the diagonal part, because splitting across the conditions for zero-divergence will yield methods that are not volume-preserving.
The goal is to split the diagonal part in divergence-free vector fields that can be solved exactly.
These observations motivate the following assumption.

\begin{assumption}
Unless otherwise stated, we will assume that the given vector field is \emph{diagonal} only (i.e.\ $\mathbf{f}=\mathbf{f}^\mathrm{diag}$) and that the off-diagonal part is zero ($\mathbf{f}^\mathrm{offdiag}=\mathbf{0}$).
\end{assumption}

In this paper we consider the case when the functions $f_i(\mathbf{x})$  are polynomials. To treat the general case with $n$ variables and given degree $d$, it is convenient to introduce a multi-index notation.

\begin{definition}
Let $\mathbf{j} = (j_1, \ldots, j_n)\in \mathbb{N}^n$ be a multi-index and let $\mathbf{x}^\mathbf{j} = x_1^{j_1} x_2^{j_2}\cdots x_n^{j_n}$ (monomial).
In addition, denote by $|\mathbf{j}| = j_1+\cdots + j_n$ the degree of the monomial.
\end{definition}

Earlier studies of splitting methods for polynomial fields have focussed on homogeneous polynomials. 
Assume that the $f_i(\mathbf{x})$ are homogeneous polynomials of degree $d$ in the variables $x_1, x_2, \ldots, x_n$, i.e.\
\begin{equation}
	\dot x_i = f_i(\mathbf{x})=\sum_{|\mathbf{j}|=d}  a^i_{\mathbf{j}}\,  \mathbf{x^j}, \qquad i = 1,\ldots, n.
		\label{eq:polyd}
\end{equation}
The case when $f_i(\mathbf{x})$ is not homogeneous is easily treated by a further splitting in homogeneous terms. 
Define $N(n,d) = {n+d-1\choose d}$. In \cite{mclachlan04egi}  it was shown that $N(n,d)$ is the number of coefficients $a^i_\mathbf{j}=a^i_{j_1, \ldots, j_n}$ (for a fixed value of $i$) of an homogeneous polynomial $f_i$ of degree $d$ in $n$ variables. The divergence-free condition is a homogeneous polynomial of degree $d-1$ because of derivation. All the coefficients of this polynomial must be identically equal to zero, because of the arbitrariness of the variables. This gives $N(n, d-1)$ conditions for volume preservation. 
In the polynomial case, there are $N(n, d-1)$ divergence-free conditions on the coefficients. In order to be volume preserving, a splitting 
must satisfy one or more conditions on the coefficients. Our approach is inspired in part from that argument. The new idea is to consider a polynomial vector field and to look at a monomial-basis expansion of the divergence function.
Each basis element will correspond exactly is associated to a condition on the coefficients of the vector field. As long as we split according to basis terms, we are guaranteed to obtain ``elementary'' monomial fields that are also divergence free. We show that all these ``elementary divergence-free vector fields'' possess $n-1$ integrals in evolution and are therefore integrable. Furthermore, we give the explicit analytical solution. This allows us to generate splitting methods for arbitrary divergence-free polynomial  fields.

We illustrate the general idea by discussing a simple 2-dimensional example. 
\begin{example} Consider the vector field
\begin{eqnarray*}
	\dot x_1 &=& a^1_ \mathbf{j} \, x_1^3 + a^1_\mathbf{k}  \, x_1^2 x_2 + a^1_\mathbf{l}  \, x_1 x_2^2\\
	\dot x_2 &=& a^2_\mathbf{j}  \, x_1^2 x_2 + a^2_\mathbf{k}  \, x_1 x_2^2 + a^2_\mathbf{l}  \, x_2^3
\end{eqnarray*}
The divergence of this vector field is
\begin{displaymath}
	p(x_1,x_2) = (3 a^1_\mathbf{j} + a^2_\mathbf{j}) x_1^2 + (2 a^1_\mathbf{k} + 2 a^2_\mathbf{k}) x_1 x_2 + ( a^1_\mathbf{l}+ 3a^2_\mathbf{l})x_2^2.
\end{displaymath}
The divergence free condition becomes the set of equations
\begin{displaymath}
	3 a^1_\mathbf{j} + a^2_\mathbf{j}=0, \qquad a^1_\mathbf{k} + a^2_\mathbf{k}, \qquad a^1_\mathbf{l}+ 3a^2_\mathbf{l} =0.
\end{displaymath}
As argued above, for each equation we obtain a divergence-free split vector fields
\begin{displaymath}
	\begin{array}{lclcl}
	\dot x_1 =a^1_ \mathbf{j} \, x_1^3	&	& \dot x_1 =  a^1_\mathbf{k}  \, x_1^2 x_2  &  & \dot x_1 = a^1_\mathbf{l}  \, x_1 x_2^2\\
	\dot x_2 =   a^2_\mathbf{j}  \, x_1^2 x_2&	&\dot x_2 = a^2_\mathbf{k}  \, x_1 x_2^2  , 	 &	&\dot x_2 =  a^2_\mathbf{l}  \, x_2^3 .  	
	\end{array}
\end{displaymath} 
Surely, each of these vector fields can be integrated. For instance, in the first vector field, one can solve for $x_1$ then substitute into the second equation to solve for   $x_2$. In the third vector field, the procedure is similar, just interchange the role of $x_1$ and $x_2$.
The second vector field can also be integrated, though not exactly by the same method.
The situation becomes far more complicated for several variables and higher order polynomial vector fields, so that, at first glance, this procedure does not yield a method that can be easily generalized.
However, observe that  the three vector fields can be written as
\begin{displaymath}
	\begin{array}{lclcl}
	\dot x_1 = a^1_ \mathbf{j} \,  x_1\cdot x_1^2 	&	& \dot x_1 = a^1_\mathbf{k}  \, x_1 (x_1 x_2) &  & \dot x_1 = a^1_\mathbf{l}  \,  x_1 \cdot x_2^2 \\
	\dot x_2 = a^2_\mathbf{j}  \, x_2 \cdot x_1^2 	&	&\dot x_2 =  a^2_\mathbf{k}  \,  x_2 (x_1 x_2) 	 &	&\dot x_2 = a^2_\mathbf{l}  \,x_2 x_2^2,   	
	\end{array}
\end{displaymath} 
i.e.\  each vector field obeys an equation of the type $\dot x_i = c_i x_i (x_1^{j_1} x_2^{j_2})$ where $x_1^{j_1} x_2^{j_2}$ are the first, the second and third monomial in $p(x_1,x_2)$. In particular, if $ (x_1^{j_1} x_2^{j_2})(t)$ is known,  $x_i$ can be obtained by integration.
\end{example}
In what follows, we shall see that this argument is generic: it yields for any number of variables and any degree of the vector field. We shall therefore identify monomials appearing in the divergence-free condition and solve for them explicitly. When the monomials, as function of time, are known, the remaining variables follow by integration of linear equations with variable coefficients.

\section{Elementary divergence-free vector fields: the polynomial case}
\label{sec:hom}
Consider a polynomial vector field
\begin{equation}
 \dot {x}_i = f_i(\mathbf{x})=\sum_{|\mathbf{k}|=1}^d a^i_{\mathbf{k}} \mathbf{x}^\mathbf{k}, \qquad i=1, \ldots, n,
 \label{eq:homd}
\end{equation}
of  degree $d$. We assume that this polynomial field is \emph{diagonal} in the sense described earlier. Diagonality implies some (mild) restrictions on the nonzero coefficients of \R{eq:homd}, namely that, for a given index $\mathbf{k}$, one has 
$ a^i_{\mathbf{k}} \not=0$ provided that $k_i\not=0$.

Rather than looking at the vector fields, it is useful to focus on the divergence polynomial and reconstruct the vector field from the divergence.

\begin{lemma}
Consider the (diagonal) polynomial vector field \R{eq:homd}. Assume that the vector field is divergence-free. Let 
\begin{displaymath}
	P_{d-1}(\mathbf{x})= \sum_{|\mathbf{j}|=0}^{d-1} p_{\mathbf{j}}{\mathbf{x}^\mathbf{j}}.
\end{displaymath}
be the divergence of \R{eq:homd}. For each multi-index $\mathbf{j}$, consider the polynomial field
\begin{equation}
 \dot{x}_i = a^i_{\mathbf{j}+\mathbf{e}_i} x_i \mathbf{x}^\mathbf{j}, \qquad 1\leq i \leq n.
 \label{eq:xi}
\end{equation}
where $\mathbf{e}_i$ the canonical unit vector in $\RR^n$ with $1$ in the $i$th position and $0$ elsewhere.
The polynomial field \R{eq:xi}  is divergence free. Moreover, \R{eq:homd} can be uniquely split in sum of vector fields of the form \R{eq:xi}.
\end{lemma}

\begin{proof}
Since monomials are basis for  the set of polynomials of degree $d$, zero-divergence of \R{eq:homd} implies that 
\begin{displaymath}
	p_{\mathbf{j}} = 0 \qquad |\mathbf{j}|=0, \ldots, d-1,
\end{displaymath}
where $p_{\mathbf{j}}$ is the coefficient of $\mathbf{x}^\mathbf{j} = x_1^{j_1}\cdots x_n^{j_n}$ in the divergence polynomial. 
For any multi-index $\mathbf{j}$, a contribution to the coefficient $p_{\mathbf{j}}$ comes from $\partial_{x_i}f_i(\mathbf{x})$, for each $1\leq i\leq n$. If the $f_i(x)$ are polynomials, such a term originates from derivation of  $x_i \mathbf{x}^\mathbf{j}=\mathbf{x}^{\mathbf{j+e}_{i}}$ with respect to $x_i$. This corresponds to the term $a^i_{ \mathbf{j}+\mathbf{e}_i} x_i \mathbf{x}^\mathbf{j}$ in $f_i(\mathbf{x})$.
This procedure picks up all the terms in $f_i(\mathbf{x})$ contributing to $p_{\mathbf{j}}$  and the condition $p_{\mathbf{j}}=0$ guarantees that \R{eq:xi} is divergence free. 
The uniqueness comes from the fact that $\int \mathbf{x}^\mathbf{j} \D x_i = \frac{1}{j_i+1}\mathbf{x}^{\mathbf{j+e}_{i}} + c h(\mathbf{x})$, where $\partial_{x_i} h(\mathbf{x})=0 $, i.e.\ $h$ is an arbitrary function of the remaining variables and does not contribute to the diagonal part of the vector field.
\end{proof}

\vspace*{10pt}
In the sequel, we will often use the short-hand notation
\begin{equation}
	\dot{\mathbf{x}} = F_\mathbf{j} (\mathbf{x})
	\label{eq:Fj}
\end{equation}
for \R{eq:xi}. Each of these divergence-free vector fields is associated to a monomial basis element, and will be called an \emph{elementary divergence-free vector field} (in short, EDFVF).

\section{Properties of the EDFVF}
\label{sec:integrals}
Without loss of generality, we can consider a EDFVF $F_\mathbf{j}$ of the form
\begin{equation}
	\dot{x}_i =  a_i x_i \mathbf{x^j}, \qquad i=1, \ldots, n,
	\label{eq:edfvf}
\end{equation}
(for simplicity, we have dropped the dependence of $\mathbf{a}$ on the actual index $\mathbf{j}$).
The divergence free condition amounts to the algebraic relation 
\begin{displaymath}
	\mathbf{a}^T (\mathbf{j+1}) =0, \qquad \mathbf{1} = (1, \ldots, 1)^T.
	\label{eq:edfcond}
\end{displaymath}
\begin{theorem}[Integrability of  monomial EDFVF]
\label{th:integrability}
The function $I_1 = \mathbf{x^{j+1}}$ is an integral of the EDFVF \R{eq:edfvf}. Moreover, if $\mathbf{b}\in \mathbb{R}^n$ is any other vector orthogonal to $\mathbf{a}$, then $\mathbf{x^b}$ is also an integral of \R{eq:edfvf}. In particular, it follows that \R{eq:edfvf} has $n-1$ independent integrals of motion, hence the EDFVF \R{eq:edfvf} is integrable.
\end{theorem}
\begin{proof}
Let $\mathbf{b}$ be an arbitrary vector orthogonal to $\mathbf{a}$, i.e.\ such that $\mathbf{a}^T \mathbf{b}=0$ (the case $\mathbf{b}_1 = \mathbf{j+1}$ is just a particular choice of $\mathbf{b}$).
We have
\begin{eqnarray*}
	\frac{\D}{\D t} \mathbf{x^b} &=& \sum_{i=1}^n b_i \dot{x}_i \mathbf{x^{b-e_i}} = \sum_{i=1}^n a_i b_i x_i \mathbf{x^j}\mathbf{x^{b-e_i}} \\
	&=& \sum_{i=1}^n a_i b_i \mathbf{x^{j+b}} = (\mathbf{a}^T \mathbf{b}) \mathbf{x^{j+b}} =0,
\end{eqnarray*}
because of the orthogonality condition. Furthermore, since $\mathbf{a} \in \RR^n$,  we can find further $n-2$ vectors orthogonal to $\mathbf{a}$ and to $\mathbf{j+1}=\mathbf{b}_1$, i.e.\ vectors $\mathbf{b}_2, \ldots, \mathbf{b}_{n-1}$ (each corresponding to the integral $I_k= \mathbf{x}^{\mathbf{b}_k}\not=0$), such that $\{ \mathbf{a}, \mathbf{b}_1, \mathbf{b}_2, \ldots, \mathbf{b}_{n-1}\}$ is an orthogonal basis of $\RR^n$. These integrals are functionally independent, i.e.\ any linear combination
\begin{equation}
	c_1 \nabla I_1 + \ldots + c_{n-1} \nabla I_{n-1} = 0 
	\label{eq:lindep}
\end{equation}
has only the trivial solution
\begin{displaymath}
	c_1 = c_2 =\ldots = c_{n-1} =0.
\end{displaymath}
To show this, note that 
\begin{displaymath}
	\nabla I_k = \left[ \begin{array}{c} (\mathbf{b}_k)_1 \mathbf{x}^{\mathbf{b}_k -\mathbf{e}_1}\\
	\vdots \\
	(\mathbf{b}_k)_n \mathbf{x}^{\mathbf{b}_k -\mathbf{e}_n}\end{array}\right]
	=
	\left[ \begin{array}{c} (\mathbf{b}_k)_1/x_1 \\
	\vdots \\
	(\mathbf{b}_k)_n/x_n \end{array}\right] I_k.
\end{displaymath}
Because of the arbitrariness of $x_1, \ldots, x_n$, condition \R{eq:lindep} is equivalent to the linear system
\begin{displaymath}
	[\mathbf{b}_1 , \ldots, \mathbf{b}_{n-1}] \mathbf{c} = \mathbf{0}
\end{displaymath}
$\mathbf{c}=(c_1, \ldots, c_{n-1})^T$, which admits only the trivial solution as the matrix columns are orthogonal, hence linearly independent.
In conclusion, each of the vector $\mathbf{b}_k$, $k=1,\ldots, n-1$, will give rise to an independent integral of motion, hence integrability of the vector field.
\end{proof}

\begin{figure}[th]
\centering
\includegraphics*[width=7cm]{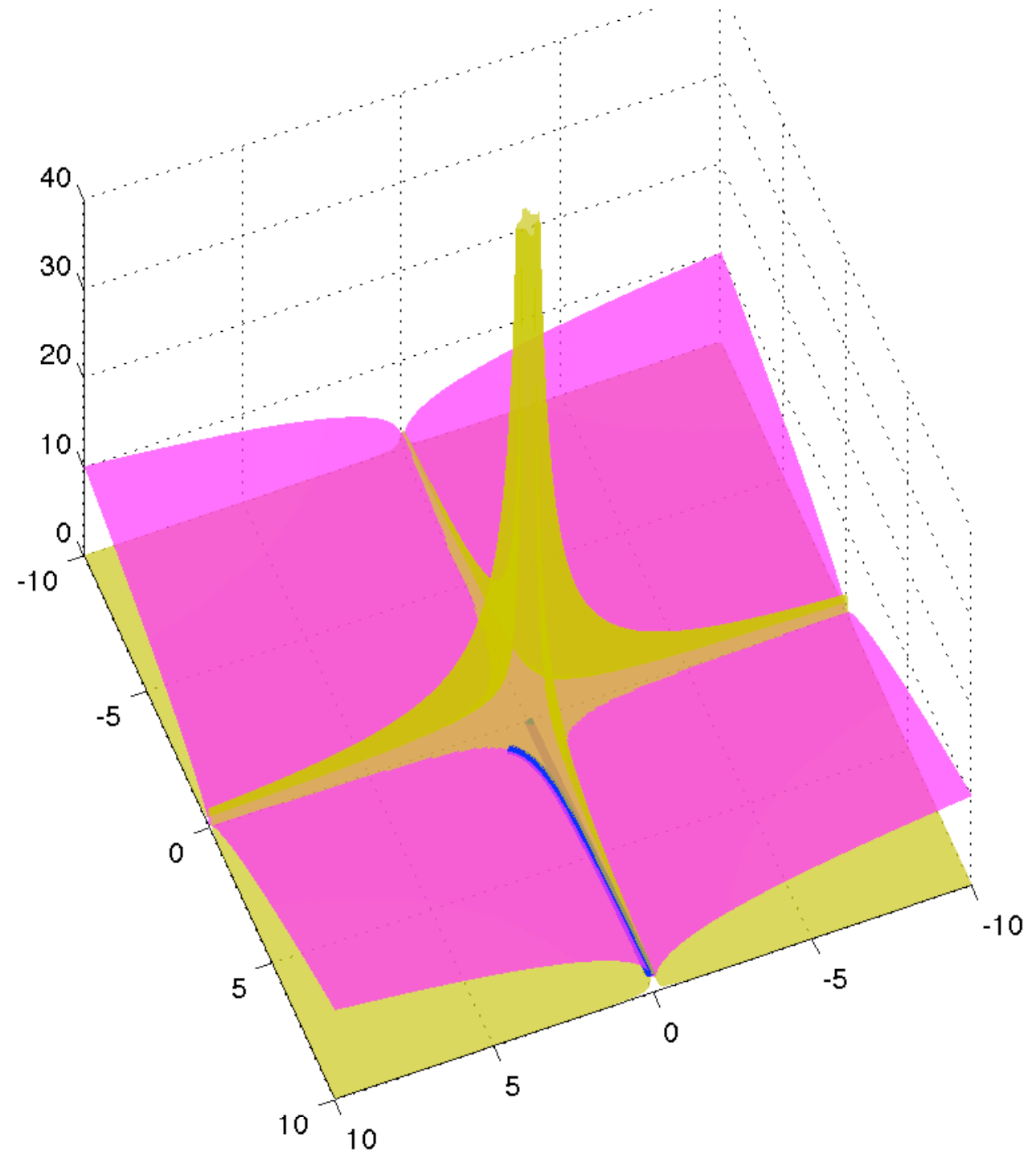}
\caption{A 3-dimensional system with $\mathbf{j} = (1,1,1)^T$, $\mathbf{a}=(-\frac53, \frac43,\frac13)^T$ and $\mathbf{b}_2 = \mathbf{a} \times (\mathbf{j+1})$ (cross product). The solution of the system lies on the intersection of the surfaces $I_1(t) = \mathbf{x^{j+1}}$ (lighter surface) and $I_2(t) = \mathbf{x}^{\mathbf{b}_2}$ (darker surface).}
\label{fig:int}
\end{figure}

Figure \ref{fig:int} shows two such integral surfaces for a 3-dimensional system \R{eq:edfvf} with $\mathbf{a}=(-\frac53, \frac43,\frac13)^T$, $\mathbf{j} = (1,1,1)^T$ and initial condition $(1,1,1)^T$. The solution (blue thick line) lies at the intersection of the two integrals of motion, $I_1$ (yellow, lighter color) and $I_2$ (purple, darker color), the latter corresponding to  $\mathbf{b}_2 = \mathbf{a} \times (\mathbf{j+1})$.

Among the $n-1$ integrals, the first one $I_1 = \mathbf{x^{j+1}}$ is more fundamental than the others, because it is a direct consequence of the divergence-free condition. In particular, it can be always used to transform the EDFVF \R{eq:edfvf} (which has nonzero diagonal part) into a new system with zero diagonal part.
\begin{corollary}
The EDFVF \R{eq:edfvf} is equivalent to the off-diagonal divergence free system
\begin{displaymath}
	\dot{x}_i =  a_i I_1(t_0) \mathbf{x^{-1+e_i}}, \qquad i=1, \ldots, n,
\end{displaymath}
where $I_1(t) = \mathbf{x^{j+1}}$ is an integral of the system.
\end{corollary}
The above system is obviously off diagonal since the $i$th component of $\mathbf{-1+e}_i$ is zero.

We have demonstrated integrability of the polynomial EDFVF by providing the existence of $n-1$ independent integral of motions. However, it is well known that the explicit knowledge of the integrals of motions does not necessarily give explicit informations about the solution. Below, we give the explicit solution for the EDFVF \R{eq:edfvf}. This is based on two fundamental observations:
\begin{itemize}
	\item $x_i(t)$ can be reconstructed by integration provided that $\mathbf{x^j}(t)$ is known,
	\item remarkably, the basis functions $\mathbf{x^j}$ \emph{always} obey the \emph{same type of solvable} differential equation, 		independently of the multi-index $\mathbf{j}$.
\end{itemize}

\begin{theorem}[Analytic solutions of monomial EDFVF]
\label{th:dcase}
Consider the EDFVF \R{eq:edfvf} associated to the monomial basis element $\mathbf{x^j}$. The associate basis function $\mathbf{x}^\mathbf{j}(t)$ obeys the differential equation
\begin{equation}
	\frac{\mathrm{d}}{\mathrm{d} t} \mathbf{x}^\mathbf{j} = c_\mathbf{j} \mathbf{x}^{2\mathbf{j}},
	\label{eq:xjdot}
\end{equation}
where $c_\mathbf{j}$ is the constant
\begin{equation}
	c_\mathbf{j} = \mathbf{a}^T\mathbf{j}.
	\label{eq:kj}
\end{equation}
For sufficiently small $\Delta t\leq \displaystyle{\frac{1}{|c_\mathbf{j}\mathbf{x}^\mathbf{j} (t_0)|}}$, 
the solution of \R{eq:edfvf} is
\begin{equation}
	x_i(t) = \begin{cases} x_i(t_0) \Big(1-c_\mathbf{j} \mathbf{x}^\mathbf{j}(t_0)(t-t_0)\Big)^{-r^i_{\mathbf{j}}} & \hbox{for $c_\mathbf{j}\not=0$}\\
	x_i(t_0) \E^{a_i \mathbf{x^j}(t_0) (t-t_0)} & \hbox{for $c_\mathbf{j}=0$},
	\end{cases}
	\label{eq:xi_sol}
\end{equation}
for $i=1, \ldots, n$, where
\begin{equation}
	r^i_{\mathbf{j}} = \frac{a_i}{c_{\mathbf{j}}}.
\label{eq:rij}
\end{equation}
\end{theorem}
\begin{proof}
By direct computation, we have:
\begin{eqnarray*}
 \frac{\mathrm{d}}{\mathrm{d} t} \mathbf{x}^\mathbf{j} &=& j_1 \dot{x}_{1} \mathbf{x}^{\mathbf{j} -\mathbf{e}_1} + \cdots + j_n \dot{x}_{n} \mathbf{x}^{\mathbf{j} -\mathbf{e}_n}\\
  &=& j_1 a_1 x_{1} \mathbf{x}^\mathbf{j}   \mathbf{x}^{\mathbf{j} -\mathbf{e}_1}  + \cdots + j_n a_n x_{n} \mathbf{x}^\mathbf{j}   \mathbf{x}^{\mathbf{j} -\mathbf{e}_n} = \sum_{i=1}^n j_{i} a_i \mathbf{x}^{2\mathbf{j}}\\
  &=& c_\mathbf{j}\mathbf{x}^{2\mathbf{j}}.
\end{eqnarray*}

The second passage follows from substituting the components of the vector field $F_{\mathbf{j}}$ in place of $\dot{x}_{i}$ and the last passage follows from the definition of $c_\mathbf{j}$. Remarkably, this differential equation is exactly the same as the first one in the quadratic case \cite{mclachlan09evp}, therefore it can be solved in the same manner, for instance, by separation of variables,
yielding 
\begin{equation}
	\mathbf{x}^\mathbf{j}(t) = \frac{\mathbf{x}^\mathbf{j}(t_0)}{1-c_\mathbf{j} \mathbf{x}^\mathbf{j}(t_0)(t-t_0)}
	= {\mathbf{x}^\mathbf{j}(t_0)}\Big( 1-c_\mathbf{j} \mathbf{x}^\mathbf{j}(t_0)(t-t_0)\Big) ^{-1},
	\label{eq:xj}
\end{equation}
in the interval $[t_0, t_0 + \Delta t]$.

Thereafter, each of the equations \R{eq:xi} is of the form $\dot{x}_i = h(t) x_i(t)$,
where $h(t) = a_i \mathbf{x}^\mathbf{j}(t)$ is a known function. As for the quadratic case (see \cite{mclachlan09evp}), the equation is linear in $x_i$, with variable coefficients, and has solution $x_i(t) = x_i(t_0) \E^{\int_{t_0}^t h(\tau) \mathrm{d} \tau}$. Hence, for $c_\mathbf{j}\not=0$, 
\begin{eqnarray*}
	x_i(t) &=& x_i(t_0) \E^{-\frac{a_i}{c_\mathbf{j}} \ln (1-{c_\mathbf{j}}\mathbf{x}_\mathbf{j}(t_0) (t-t_0))} = x_i(t_0) \Big(\E^{\ln (1-{c_\mathbf{j}}\mathbf{x}_\mathbf{j}(t_0) (t-t_0))}\Big)^{-r^i_{ \mathbf{j}}} \\
	&=& x_i(t_0) \Big(1-c_\mathbf{j} \mathbf{x}^\mathbf{j}(t_0)(t-t_0)\Big)^{-r^i_{\mathbf{j}}},
\end{eqnarray*}
where $r^i_{\mathbf{j}}$ is defined as in \R{eq:rij}. The case $c_\mathbf{j}=0$ follows in a similar manner. This completes the proof of the statement.
\end{proof} 

\begin{example}
Consider the divergence-free differential equation
\begin{equation}
	\begin{array}{lcl}
	\dot x_1 &=& x_1^4 + x_3\sin x_1 \cos x_2 + 4 x_2 x_3\\
	\dot x_2 &=& -\frac12 x_1^3 x_2-\frac12 x_3 \cos x_1 \sin x_2 -x_1 \\
	\dot x_3 &=& -\frac12 x_1^3 x_3-\frac14 x_3^2 \cos x_1 \cos x_2 + x_1^2 x_2.
	\end{array}
	\label{eq:poly_ex}
\end{equation}
This system  consists of a diagonal part and a off-diagonal part. In many applications, it is common to approximate such vector fields by a polynomial field of low degree (like linear, quadratic, etc.). This is typically done by using a Taylor expansion and truncating to a certain degree $d$.
The choice $d=5$ gives the system
\begin{equation}
	\dot{\mathbf{x}} = \mathbf{f}^\mathrm{diag} (\mathbf{x}) + \mathbf{f}^\mathrm{offdiag}(\mathbf{x}),
	\label{eq:poly_ex_split}
\end{equation}
where 
\begin{displaymath}
	\mathbf{f}^{\mathrm{diag}} (\mathbf{x}) = \left[ \begin{array}{c}
	x_1 x_3 + x_1^4 -\frac16 x_1^3 x_3 -\frac12 x_1 x_2^2 x_3\\
	-\frac12 x_2 x_3-\frac12 x_1^3 x_2 +\frac1{12} x_2^3 x_3 +\frac14 x_1^2 x_2 x_3\\
	-\frac14 x_3^2-\frac12 x_1^3 x_3 +\frac18 x_1^2 x_3^2 +\frac18 x_2^2 x_3^2
	\end{array}
	\right],
	\quad
	\mathbf{f}^\mathrm{offdiag}(\mathbf{x}) = \left[ \begin{array}{c}
		4 x_2 x_3 \\
		-x_1 \\
		x_1^2 x_2
	\end{array}
	\right].
\end{displaymath}
The off-diagonal part can be treated using canonical shears.
For the diagonal, observe that the divergence-free condition for \R{eq:poly_ex_split} is a polynomial in $x_3, x_1^3, x_1^2 x_3, x_2^2 x_3$.
These correspond to the multi-indices  $\mathbf{j}=(0,0,1)$, $\mathbf{k} = (3,0,0)$,  $\mathbf{l} = (2,0,1)$, and $\mathbf{m} = (0,2,1)$.
Each of these multi-indices is associated to an EDFVF.  For instance, for $\mathbf{l}$,
one has $c_{\mathbf{l}} =  -\frac5{24}$ and  \R{eq:rij} gives
$r^1_{ \mathbf{l}} = \frac45, r^2_{\mathbf{l}} = -\frac65,  r^3_{\mathbf{l}} =-\frac35$.
Assuming that the solution is known at $t_k$, we obtain 
\begin{eqnarray*}
	x_1(t) &=& x_1(t_k) (1+ \frac{5}{24} x_1^2(t_k)x_3(t_k) (t-t_k))^{-4/5}, \\
	x_2(t) &=& x_2(t_k) (1+ \frac{5}{24} x_1^2(t_k)x_3(t_k) (t-t_k))^{6/5},\\
	x_3(t) &=& x_3(t_k) (1+ \frac{5}{24} x_1^2(t_k)x_3(t_k) (t-t_k))^{3/5},
\end{eqnarray*}
for $t \in [t_k, t_{k+1}]$.
A similar treatment applies to the EDFVF associated to the other multi-indices.
\end{example}

\section{Extensions to the case of negative (Laurent polynomials), rational and real multi-indices bases}
\label{sec:laurent}
In this section, we investigate the extension of the theory to treat the case of splitting in terms of the form $\mathbf{x^j}$ with $\mathbf{j}\in \mathbb{Z}^n$, or, more generally, $\mathbf{j} \in \mathbb{R}^n$.

We assume that the divergence of the vector field can be expressed in the form
\begin{equation}
	P(x_1, x_2, \ldots, x_n) = \sum_{\mathbf{j}\in \mathcal{J}}p_\mathbf{j} \mathbf{x}^\mathbf{j},
	\label{eq:divLaurent}
\end{equation}
where the $\mathcal{J}\subset \mathbb{F}^n$ (where  $\mathbb{F}= \mathbb{Z},\mathbb{R}$) is discrete finite set of multi-indices.
\begin{lemma}
\label{th:jine-1}
	Assume that the divergence free vector field \R{eq:vf} has components 
	\begin{equation}
	f_i (\mathbf{x}) = \sum_{\mathbf{l}\in \mathcal{L}} a^i_{\mathbf{l}} \mathbf{x}^{\mathbf{l}}, 
	\label{eq:gen_mon}
	\end{equation}
where $\mathcal{L} \subset \mathbb{F}^n$ (discrete set), $i=1,\ldots, n$. Then, the divergence of $\mathbf{f}$ is a generalized polynomial of the form \R{eq:divLaurent}.
Moreover, either the indices $\mathbf{j}$ in \R{eq:divLaurent} have components $j_i \not=-1$, or, if some $j_i = -1$, then equation $i$ does not contribute to this multi-index.
\end{lemma}
\begin{proof}
Terms contributing to the divergence originate from $\partial f_i /\partial x_i$. 
Any $\mathbf{x^l} = x_1^{l_{1}}\cdots x_i^{l_{i}} \cdots x_i^{l_{n}}$ contributes $a^i_{\mathbf{l}}  l_{i} x_1^{l_{1}}\cdots x_i^{l_{i}-1} \cdots x_n^{l_{n}}= a^i_{\mathbf{l}}  l_i \mathbf{x^j}$, to \R{eq:divLaurent}, $\mathbf{j = l-e}_i$, and it is easily seen that linear combinations of these terms  yield a ``generalized'' polynomial of the form \R{eq:divLaurent}. As far as the second part of the statement is concened, we see that $j_{i}=-1$ (yielding a term of the type $x_i^{-1}$ in the divergence) implies $l_i=0$ (this contribution is always identically zero). This rules out that a contribution to the term $x_i^{-1}$ can come from the $i$th equation but it does not rule out such a contribution to come from another equation.  
\end{proof}

As a matter of facts, a contribution to the term $x_1^{j_{1}} \cdots x_i^{-1}\cdots x_n^{j_{n}}$ from equation $i$ can occur if and only if a term of the type $x_1^{j_{1}} \cdots \log x_i \cdots x_n^{j_{n}}$ appears in $f_i$, which is ruled out from the hypotesis that $f_i$ is of a Laurent polynomial type.

The results in Theorem~\ref{th:integrability}-\ref{th:dcase}  can now be extended with minor modifications to treat the generalized polynomial case.

\begin{theorem}
\label{th:Laurent}
Let be given the differential equation \R{eq:vf}, where $\mathbf{f}$ component-wise satisfies \R{eq:gen_mon}. Let
\begin{equation}
	0 = \sum_{\mathbf{j}\in\mathcal{J}} p_\mathbf{j} \mathbf{x}^\mathbf{j}
	\label{eq:divfree}
\end{equation}
be the corresponding divergence-free condition. 
For each multi-index $\mathbf{j} \in \mathcal{J}$, there exists the associated elementary divergence-free vector field 
\begin{equation}
	\dot x_i = (1-\delta_{i,l})	a^i_{ \mathbf{j+e}_i} x_i \mathbf{x^j}, \qquad i=1,\ldots, n, \qquad j_{l}=-1,
	\label{eq:xidotdelta}
\end{equation}
where $\delta_{i,l}$ is the Kronecker delta and $l$ is any index such that $j_l=-1$. 
The elementary divergence free  vector  fields \R{eq:xidotdelta} can be written in the form
\begin{displaymath}
	\dot x_i = a_i x_i \mathbf{x^j},
\end{displaymath}
(disregarding the dependance on the index $\mathbf{j})$, they are volume preserving by construction and possess $n-1$ integrals of motions, hence they are integrable and their solution is explicitly given as in \ref{th:dcase}.

\end{theorem}
\begin{proof}
The proof is very similar to the monomial case. The only difference is the presence of the term $\delta_{i,l}$, taking care of excluding the contribution of the $l$th coefficient in $a_{ \mathbf{j} +\mathbf{e}_{i}}$ (because this contributes to a off-diagonal term, see above lemma) for the specific $\mathbf{j}$ with $j_l = -1$. 
\end{proof} 

\section{Numerical examples}
\label{sec:num}
We present some interesting applications to volume-preserving vector fields for a quadratic and a cubic Stokes flow. 
In our literature search, we didn't find any relevant examples of Laurent volume-preserving fields, and we will test the methods on a artificial example, for the sake of illustration.

\subsection{Quadratic Stokes flow}
We consider a quadratic volume-preserving system introduced in \cite{MR2433381} to study the distruction of adiabatic invariance under separatrix crossing,
\begin{equation}
\begin{meqn}
	\dot x_1 &=& -8x_1 x_2 + \varepsilon x_3,\\
	\dot x_2 &=& 11 x_1^2 +3x_2^2 + x_3^2 - 3,\\
	\dot x_3 &=& 2x_3x_2 - \varepsilon x_1.
\end{meqn}
\label{eq:QuadStokes}
\end{equation}
The system is integrable for $\varepsilon=0$. When $\varepsilon \not=0$ the system is no longer integrable but solutions inside the unit sphere remain bounded in the sphere. We consider the non-integrable case and split the system into a diagonal and off-diagonal part,
\begin{displaymath}
 	\mathbf{f}^\mathrm{diag}: \quad \begin{array}{lcl}
	\dot x_1 &=& -8x_1x_2 \\
	\dot x_2 &=& 3x_2^2 \\
	\dot x_3 &=& 2x_3x_2  \end{array}
	\hspace*{30pt}
	\mathbf{f}^\mathrm{offdiag}: \quad \begin{array}{lcl}
	\dot x_1 &=& \varepsilon x_3,\\
	\dot x_2 &=& 11 x_1^2 + x_3^2 - 3,\\
	\dot x_3 &=& - \varepsilon x_1. \end{array}
\end{displaymath}
The diagonal part corresponds to the index $\mathbf{j} = (0,1,0)$, which is solved by a single flow as in Prop.~\ref{th:dcase}, with $c_\mathbf{j} = [0,1,0] [-8,3,2]^T=3$ (scalar product) and $r_\mathbf{j} = \frac13 \times [-8,3,2]^T$.

\begin{figure}[t]
\centering
\includegraphics*[width=.9\textwidth]{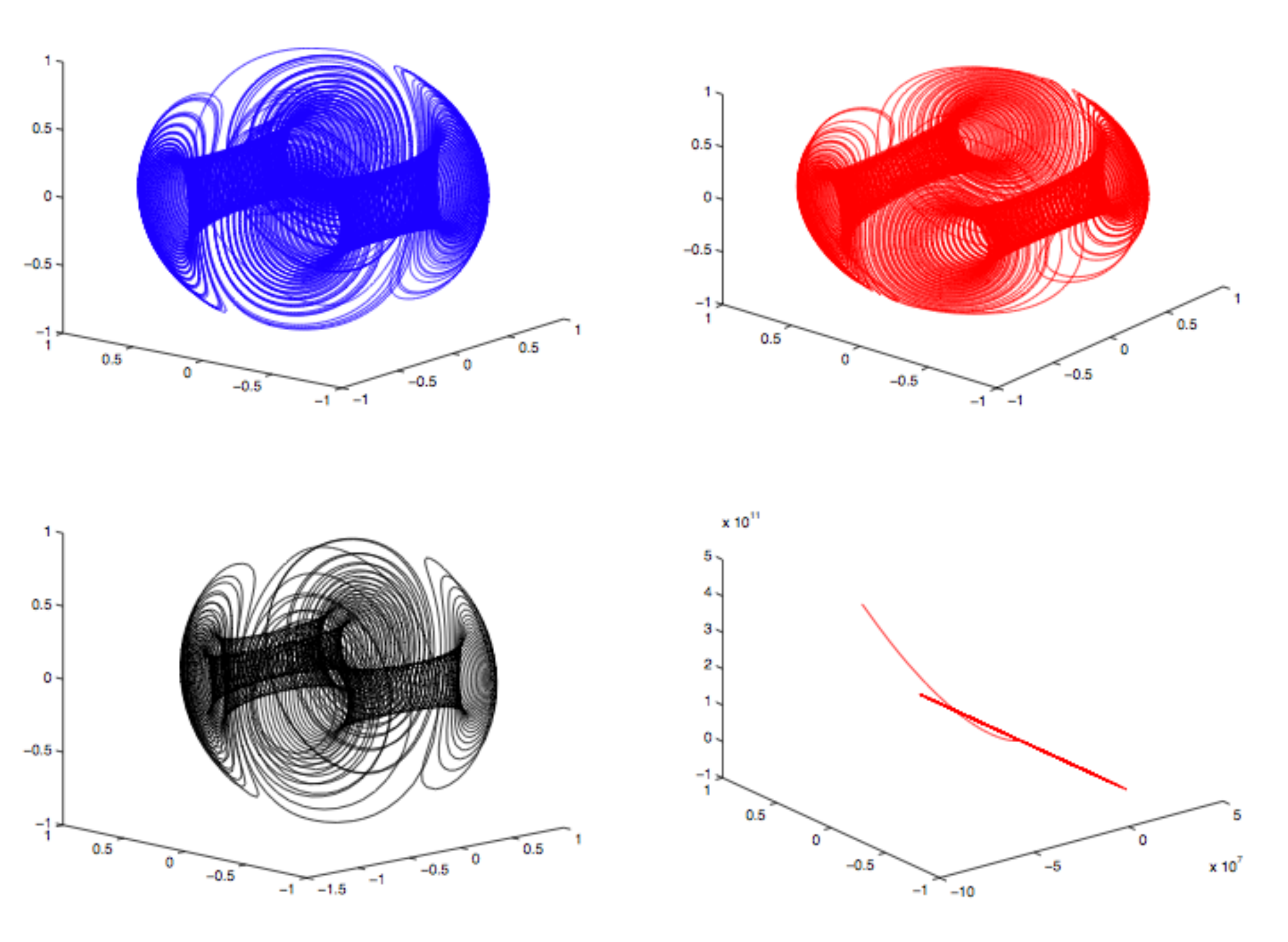}
\caption{The quadratic Stokes flow \R{eq:QuadStokes} with $\varepsilon=0.1$. Initial condition $\mathbf{x}_0=[ 0,0,0.96]^T$, integration time $T=500$.  Top left: volume-preserving method, order 2, implemented with stepsize $h=0.01$. Top right: \texttt{ode45}, using step size control, with options \texttt{RelTol=1e-6} (default value \texttt{1e-3}). Bottom left: \texttt{ode45}, using step size control, default implementation, until $T=220$, the method becomes unstable at $T=2.2887$e+2. Bottom right: same as bottom left, but letting $T=250$.}
\label{fig:QuadStokes}
\end{figure}
\begin{figure}[ht]
\centering
\includegraphics*[width=.95\textwidth]{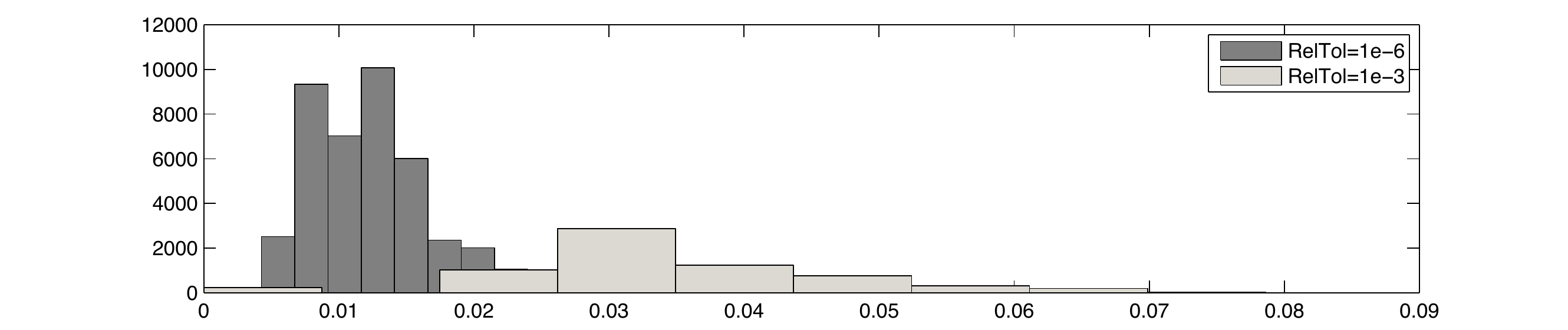}
\caption{Histogram of the step sizes used by \texttt{ode45} to meet the prescribed tolerance for the quadratic volume-preserving flow \R{eq:QuadStokes}. Dark gray: \texttt{RelTol=1e-6}. Light gray: \texttt{RelTol=1e-3}.}
\label{fig:hist45}
\end{figure}

In Figure~\ref{fig:QuadStokes} we show the numerical trajectories corresponding to $\varepsilon = 0.1$ and initial condition $\mathbf{x}_0=[ 0,0,0.96]^T$, for $T=500$. The second-order volume-preserving method (top left) is implemented as described in Algorithm~1, with a constant step size $h=0.01$, and gives a nice and bounded trajectory. The first-order method gives very similar results. For comparison, we use Matlab's \texttt{ode45} method, that is explicit, fourth-order, and uses step size control. The \texttt{ode45} method is not volume preserving unless the solution is computed to machine accuracy. 
The standard implementation of \texttt{ode45} (with stepsize control) becomes unstable at $T=2.2887$e+2 (see illustration bottom right). In the bottom left picture we display the integration up to $T=220$. To obtain a result similar to the volume-preserving method, it is necessary to decrease the error tolerance. The top right plot illustrates the results by \texttt{ode45} with options set to \texttt{RelTol=1e-6} (default value \texttt{1e-3}). 

Figure~\ref{fig:hist45} shows the histogram for the stepsize chosen by \texttt{ode45} to meet the required tolerance. For the smallest error tolerance, the average step size is around $h=0.01$, which justifies the choice for our explicit volume-preserving methods. Numerical tests revealed that our methods were stable up to a choice $h \approx 0.05$. 
We performed also long time integration, $T=100000$, with $h=0.05$, and the solution still stayed bounded for the volume-preserving method.

As far as computational time is concerned, our methods have the advantage of being explicit.  The speedup, with respect to \texttt{ode45}  implemented with the option \texttt{RelTol=1e-6} for stability, is approximatly $3.2$ for the first-order volume-preserving method and $3.15$ for the second-order method, over an average of 100 runs with the same initial condition as above, see Table~\ref{tab:1}.

\begin{table}[th]
\centering
\begin{tabular}{l||r|r|r||}
 & \texttt{ode45} & vol.\ pres., order 1 & vol.\ pres., order 2 \\ \hline \hline
 cpu time [sec] & 3.8687 & 1.2099 & 1.2260\\ \hline
speed up  	& 1 & 3.1976 & 3.1557 \\ \hline
\end{tabular}
\caption{CPU time (in seconds) for \texttt{ode45} with \texttt{RelTol=1e-6} and the explicit volume-preserving splitting methods. The values, for indication only, are computed as an average of 100 runs on an standard laptop (MacBook Pro) with initial condition $\mathbf{x}_0=[ 0,0,0.96]^T$ and $T=500$.}
\label{tab:1}
\end{table}

\begin{figure}[th]
\centering
\includegraphics*[width=.9\textwidth]{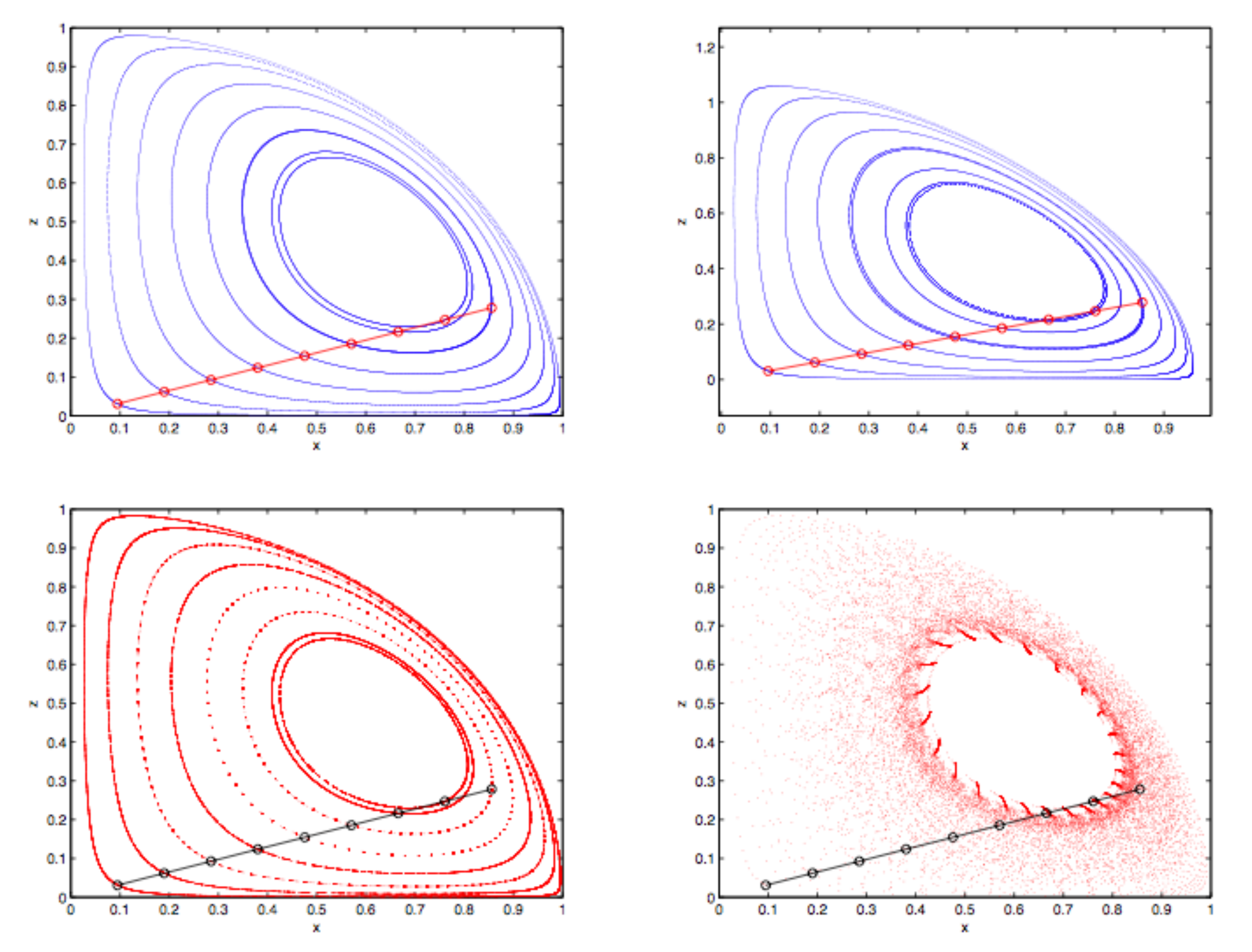}
\caption{The cubic Stokes flow \R{eq:CubicStokes} with $\mathbf{w}=\mathbf{0}$ (integrable case). Simulation of trajectories for $T=20000$
for various initial conditions (the circles on the line denote the starting value).
Top left: volume-preserving method, order 2, implemented with stepsize $h=0.01$. Top right: volume-preserving method, order 1, implemented with stepsize $h=0.01$. Bottom left: \texttt{ode45}, using step size control, with \texttt{RelTol=1e-6}. Bottom right: \texttt{ode45}, default implementation (\texttt{RelTol=1e-3}). }
\label{fig:CubicStokes1}
\end{figure}

\subsection{Cubic Stokes flow}
We consider a cubic volume-preserving flow which has been used for the investigation of the kinematics of bounded steady Stokes flows, in particular, to investigate the streamlines inside a neutrally buoyant spherical drop immersed in a genera linear flow \cite{stone91csi}. This system is described by the divergence-free differential equation
\begin{displaymath}
\dot{\mathbf{x}} = \frac12 [(5 r^2 -3) E \mathbf{x} - 2\mathbf{x} (\mathbf{x}^T E \mathbf{x})] + \frac12 \mathbf{w} \times \mathbf{x},
\end{displaymath}
where $r^2 = \mathbf{x}^T\mathbf{x}$, $\mathbf{w}$ is the vorticity vector and $E$ is the symmetric traceless rate-of-strain tensor of the external motion, that, in dimensionless terms, takes the form
\begin{displaymath}
	E = \left( \begin{array}{ccc} 1/(1+\alpha) &0 & 0 \\
		0 & \alpha/(1+\alpha) &0  \\
		0 & 0 & -1\end{array} \right),
\end{displaymath}
where $\alpha = E_{22}/E_{11}$, and $\times$ is the usual cross product.

In Cartesian coordinates, the dynamical system for investigating fluid particle motion takes the dimensionless form:
\begin{equation}
\begin{meqn}
 \dot x_1 &=& \frac12 \left[ (5r^2 -3) \frac{x_1}{1+\alpha} -2 x_1 \left( \frac{x_1^2}{1+\alpha} + \frac{\alpha x_2^2}{1+\alpha} - x_3^2 \right) \right] + \frac12 (w_2 x_3 - w_3 x_2) \\
 \dot x_2 &=& \frac12 \left[ (5r^2 -3) \frac{\alpha x_2}{1+\alpha} -2 x_2\left( \frac{x_1^2}{1+\alpha} + \frac{\alpha x_2^2}{1+\alpha} - x_3^2 \right) \right] + \frac12 (w_3 x_1 - w_1 x_3) \\
 \dot x_3 &=& \frac12 \left[ -(5r^2 -3)x_3 -2 x_3 \left( \frac{x_1^2}{1+\alpha} + \frac{\alpha x_2^2}{1+\alpha} - x_3^2 \right) \right] + \frac12 (w_1 x_2 - w_2 x_1). \\
\end{meqn}
\label{eq:CubicStokes}
\end{equation}
The case $\mathbf{w=0}$ is integrable. The generic case $\mathbf{w \not=0}$ is not integrable, but also in this case, solutions with initial condition in the unit sphere are bounded to the unit sphere, see \cite{stone91csi} for more details.

In the following experiments, the cubic Stokes flow \R{eq:CubicStokes} was computed numerically by volume-preserving methods and Matlab's \texttt{ode45} method. 
Figure \ref{fig:CubicStokes1} shows the comparison of the trajectories by the  volume-preserving methods (order $1$ and $2$) and Matlab's \texttt{ode45} method with options \texttt{RelTol=1e-6} and \texttt{RelTol=1e-3} 
for $\mathbf{w}=\mathbf{0}$. The volume-preserving methods 
preserve the qualitative behaviour for various initial conditions (denoted by small circles) in Figure \ref{fig:CubicStokes1}.
However, Matlab's \texttt{ode45} method, with standard options control, has a dissipative behaviour. To obtain the same visual result as the volume-preserving methods, one has to reduce the tolerance on the error, for instance set \texttt{RelTol=1e-6}.  
See the picture at the bottom of Figure \ref{fig:CubicStokes1} for more details.

In the numerical investigation presented below (see Figure~\ref{fig:CubicStokes2}), we study the effect of changing both the orientation and the magnititude of $\mathbf{w}$. Without loss of generality, we will take 
$\mathbf{w}=(w_1,0,w_3)$. The orientation of $\mathbf{w}$, measured from the $x_3$-axis in the $(x_1,x_3)$ plane, will be denoted by $\varTheta$. 
To describe the three-dimensional particle paths, we present Poincar{\'e} sections through the $(x_1,x_3)$ plane. 
We compute the particle paths by solving \R{eq:CubicStokes} numerically  using  a volume-preserving method (order 2), Matlab's \texttt{ode45} method with options \texttt{RelTol=1e-6} and standard setup (\texttt{RelTol=1e-3}) respectively. For more detail about the experiments design, see \cite{stone91csi} for reference. 

To start with, we examine the case $\varTheta=0.275\pi$ and $\|\mathbf{w}\|=1.5$ illustrated in the first row of Figure \ref{fig:CubicStokes2}. The left column is the Poincar\'{e} section obtained by the volume-preserving method (order 2), while the middle and right columns are obtained by Matlab's \texttt{ode45} method with options \texttt{RelTol=1e-6} and standard implementation (\texttt{RelTol=1e-3}) respectively. The particle parth Poincar\'{e} section, computed by our volume-preserving method and Matlab's \texttt{ode45} method with options \texttt{RelTol=1e-6} are visually identical to the figure showed in \cite{stone91csi} (see Figure 3 in \cite{stone91csi} for detail), while  Matlab's \texttt{ode45} method with standard options \texttt{RelTol=1e-3} gives a totally wrong section. 

We next study the Poincar\'{e} section for a fixed orientation of the vorticity vector, but  increasing magnitude of the vorticity. The second and third rows of Figure \ref{fig:CubicStokes2} show simulations for an orientation $0.2\pi$ from the $x_3$-axis by increasing $\|\mathbf{w}\|$ from $2.5$ to $4.0$. As the magnitude of $\mathbf{w}$ increases, islands are created and destroyed 
until $\|\mathbf{w}\|=1.4$. For larger magnitudes of the vorticity vector, the streamlines structure is again modified and becomes much more complex \cite{stone91csi}. 

The last two rows in the Figure \ref{fig:CubicStokes2} show the cases where the vorticity orientation increases from $0.02\pi$ to $0.4\pi$ with $\|\mathbf{w}\|=2.0$ fixed. 

For all the experiments illustrated in  Figure~\ref{fig:CubicStokes2}, the simulations have been performed with $T=20000$ and initial condition $\mathbf{x}_0= [-0.1689,0,-0.0437]^T$. The volume-preserving method (order 2) is implemented with constant step size $h=0.01$ (the last row with step size $h=0.005$). Each row corresponds to different value of $\mathbf{w}$ as explained above.

From Figure \ref{fig:CubicStokes2}, we can immediately observe that the particle parths in the Poincar\'{e} sections by the volume-preserving method behave qualitatively similarly to the exact solution, while, for non volume-preserving methods (like Matlab's \texttt{ode45}), stronger restrictions on the step size are required to obtain the desired dynamic result. 

\begin{figure}[htp]
\centering
\includegraphics*[width=\textwidth]{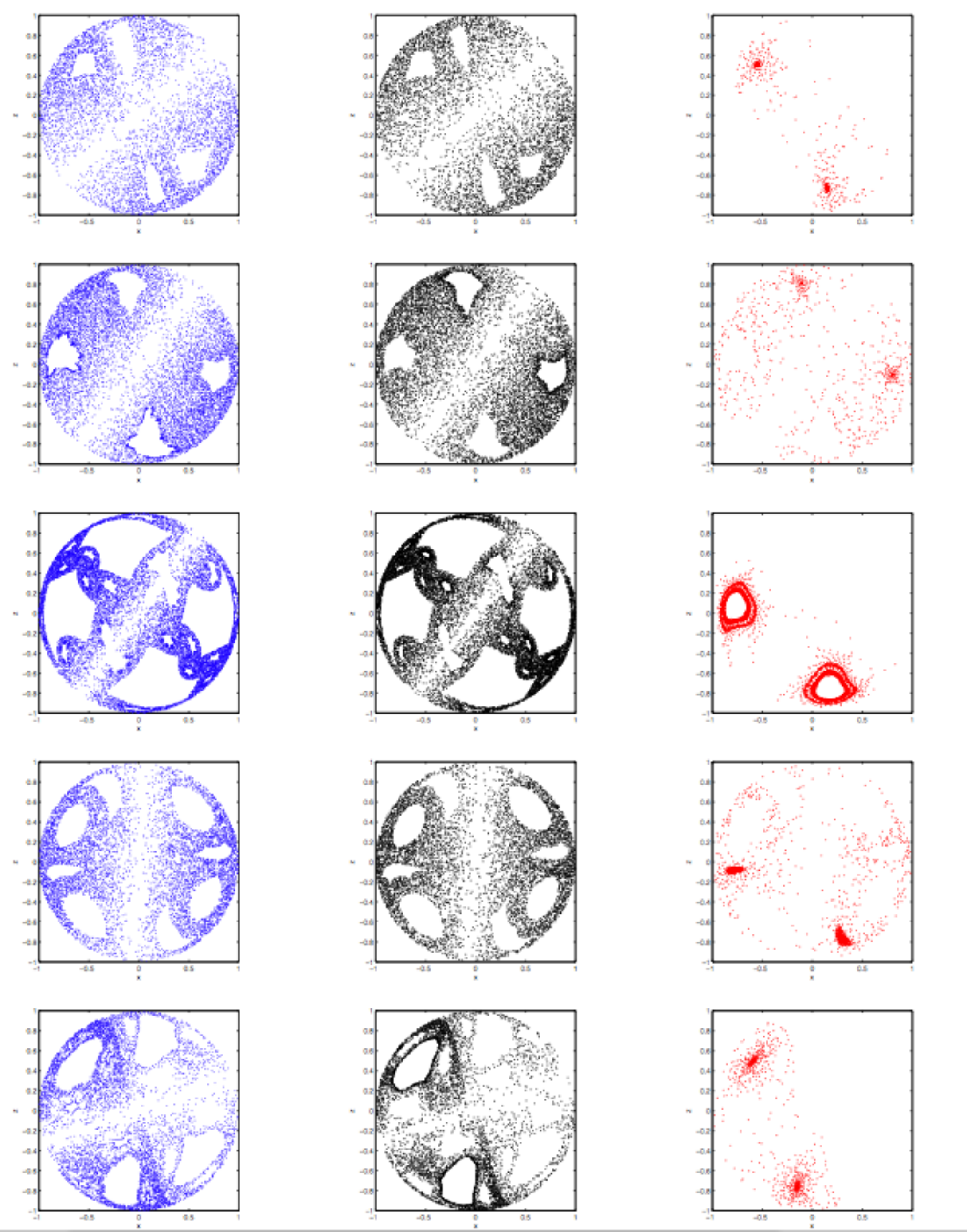}
\caption{The cubic Stokes flow \R{eq:CubicStokes} with $\mathbf{w}\not=0$ (non-integrable case). Simulation of trajectories for $T=20000$ with initial condition $\mathbf{x}_0= [-0.1689,0,-0.0437]^T$.
Left column: volume-preserving method, order 2, stepsize $h=0.01$ (with exception of last row, for which $h=0.005$). Middle column: \texttt{ode45} with option \texttt{RelTol=1e-6}. Right column:  \texttt{ode45}, standard implementation (\texttt{RelTol=1e-3}). Each row corresponds to a different value of $\mathbf{w}=(w_x,0,w_y)$. From top to bottom: $\|\mathbf{w}\|=1.5, \varTheta=0.275\pi$; $\|\mathbf{w}\|=2.5, \varTheta=0.2\pi$;  $\|\mathbf{w}\|=4.0, \varTheta=0.2\pi$;
$\|\mathbf{w}\|=2.0, \varTheta=0.02\pi$; $\|\mathbf{w}\|=2.0, \varTheta=0.4\pi$.}
\label{fig:CubicStokes2}
\end{figure}

\subsection{Laurent polynomial example}
In order to illustrate our method for Laurent-type  volume-preserving fields, we consider the following example:
\begin{equation}
\begin{meqn}
	\dot x_1 &=& 3x_1^{-2}x_2^2 + 2x_1^3x_2^{-3},\\
	\dot x_2 &=& 2x_1^{-3}x_2^3 +3x_1^2x_2^{-2}.
\end{meqn}
\label{eq:Laurent}
\end{equation}
According to the approach presented in this paper,  we split the above equation (\ref{eq:Laurent}) into two vector fields $F_{1}$ and $F_{2}$, corresponding to the multi-index $\mathbf{j}_{1}=(-3,2)$ and 
$\mathbf{j}_{2}=(2,-3)$,
\begin{displaymath}
 	F_{1} : \quad \begin{array}{lcl}
	\dot x_1 &=& 3x_1^{-2}x_2^2 \\
	\dot x_2 &=& 2x_1^{-3}x_2^3
        \end{array}
	\hspace*{30pt}
	F_{2} : \quad \begin{array}{lcl}
	\dot x_1 &=& 2x_1^3x_2^{-3}\\
	\dot x_2 &=& 3x_1^2x_2^{-2}.
	\end{array}
\end{displaymath}

We can integrate the  vector fields $F_{1}$ and $F_{2}$ explicitly according to Theorem~\ref{th:Laurent}. By choosing different initial values around $[-0.5689,0.0437]$, we compare the results obtained by 
our method to Matlab's \texttt{ode45} method with options \texttt{RelTol=1e-8}, \texttt{RelTol=1e-8}, \texttt{RelTol=1e-10} and \texttt{RelTol=1e-6} and standard implentation \texttt{RelTol=1e-3} respectively.
\begin{figure}[tp]
\centering
\includegraphics*[width=.9\textwidth]{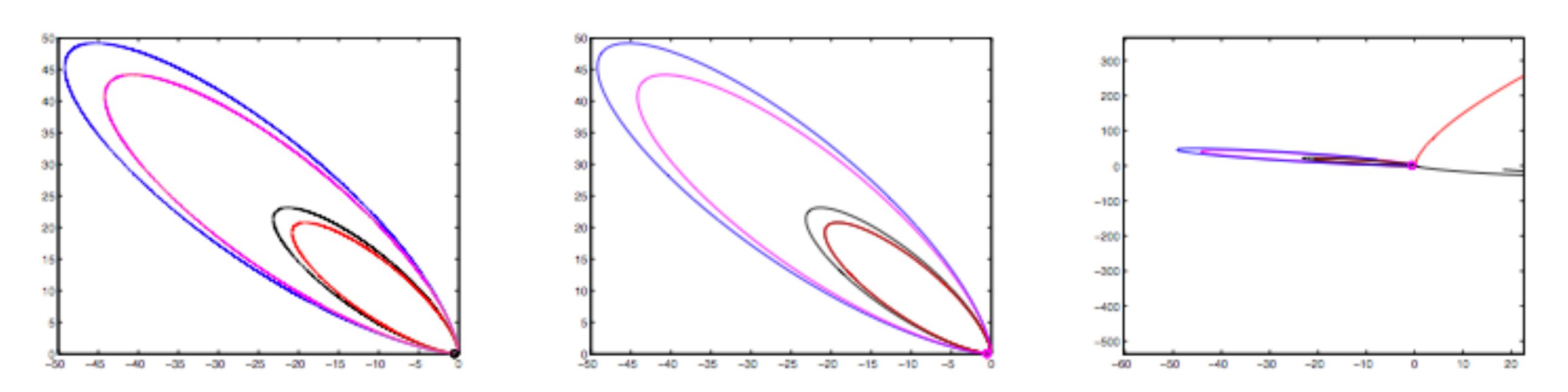}
\caption{Laurent polynomials problem (\ref{eq:Laurent}) with different initial values. Left: volume-preserving method, order 2, implemented with stepsize $h=0.001$. Middle: \texttt{ode45}, using
 stepsize control, with options \texttt{RelTol=1e-8}, \texttt{RelTol=1e-8}, \texttt{RelTol=1e-10} and \texttt{RelTol=1e-6}.   Right: \texttt{ode45}, using step size control, with standard options \texttt{RelTol=1e-3}.}
\label{fig:L1}
\end{figure}
In Figure~\ref{fig:L1}, we show the numerical results for the volume-preserving method and  Matlab's \texttt{ode45} method, for initial values $[-0.5689+k\Delta,0.0437+l\Delta]^T$, $k,l=0,1$, $\Delta = 0.02$ (a small square).
The standard Matlab's \texttt{ode45} implementation diverged for early values of $T$ (right column). Thus, to obtain the correct dynamical behaviour, stronger restrictions on the error had to be imposed (middle column): a relative tolerance  \texttt{RelTol=1e-8} gave results similar to the volume-preserving method for two of the initial conditions, for another one, the error control needed to be sharpened (\texttt{RelTol=1e-10}), while for the last one, it was sufficient to use \texttt{RelTol=1e-6}.

\begin{figure}[ht]
\centering
\includegraphics*[width=.9\textwidth]{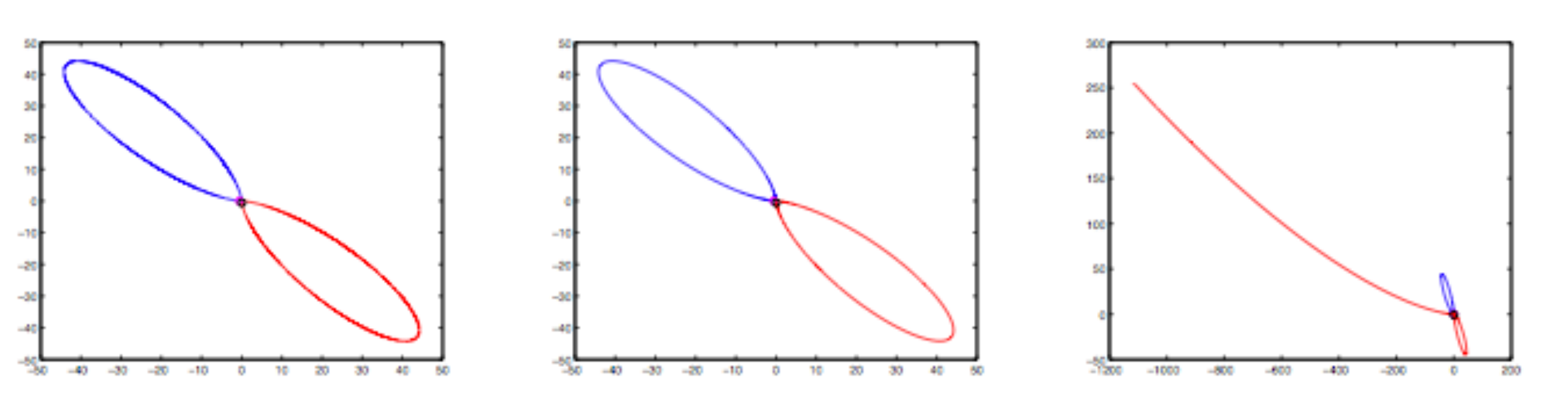}
\caption{Laurent polynomials problem (\ref{eq:Laurent}) with initial conditions $(0.0437,-0.5489)$, $(-0.5489,0.437)$ in the different Octant in the plane. Left: volume-preserving method, order 2, implemented with stepsize $h=0.001$. Middle: \texttt{ode45}, using
 stepsize control, with options \texttt{RelTol=1e-8}.   Right: \texttt{ode45}, using stepsize control, with standard implementation \texttt{RelTol=1e-3}.}
\label{fig:LL1}
\end{figure}

\begin{figure}[ht]
\centering
\includegraphics*[width=0.95\textwidth]{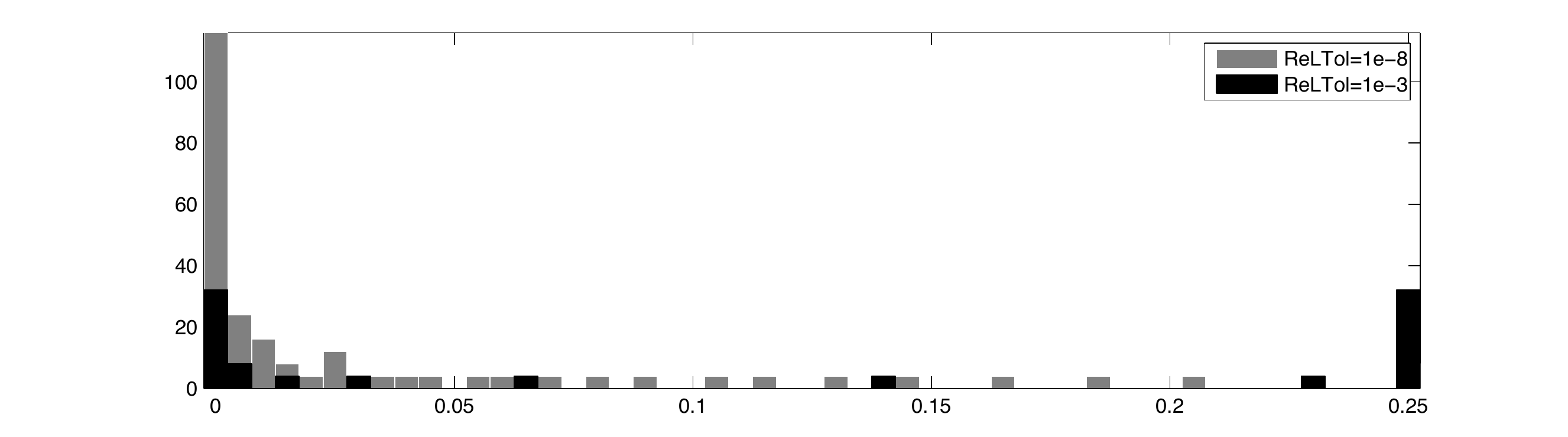}
\caption{Histogram of the stepsizes used by \texttt{ode45} to meet the prescribed tolerance for \R{eq:Laurent} in $[0,10]$. Gray: \texttt{RelTol=1e-8}. Black: \texttt{RelTol=1e-3}}
\label{fig:L2}
\end{figure}
 
 Figure $\ref{fig:LL1}$ shows the trajectories obtained with initial values $[0.0437,-0.5489]^T$ and $[-0.5489,0.437]^T$ (opposite octants) using the volume-preserving method (order 2) 
with stepsize $h=0.001$ and Matlab's \texttt{ode45} methods with options \texttt{RelTol=1e-8} and  standard implementation. 
Matlab's \texttt{ode45} method with standard implementation was again unstable, and the trajectories diverged before the end of the integration interval was reached. 
The trajectories of the system \R{eq:Laurent} are very sensitive with respect to the initial condition (ill conditioned problem) and small changes of initial values may cause tremendous difference in the actual trajectories.  This type of problems is well known to affect the numerical methods by imposing constraints on the step size (small step size must be used to remain sufficiently close to the true trajectory). Also our volume-preserving methods were affected by this, as well as Matlab's \texttt{ode45}, but they required milder restrictions.

Figure~\ref{fig:L2} shows the histogram for the stepsize chosen by \texttt{ode45} with \texttt{RelTol=1e-8} and standard implementation (\texttt{RelTol=1e-3}) respectively, to meet the required tolerance in $[0,10]$.
For the case \texttt{ode45} with options \texttt{RelTol=1e-8}, the average stepsize is the interval $(0,0.005)$. Looser restrictions gave diverging solutions. For comparison, the volume-preserving methods could use $h=0.25$ and still remain bounded.
However, for longer time integrations, also the volume-preserving methods had to reduce the step size to avoid divergence.

\section{Conclusion and further remarks}
\label{sec:conclusion}
In this paper we have presented new, explicit, volume-preserving splitting methods for Laurent polynomial divergence-free vector fields.
The methods rely on a decomposition of the divergence into a monomial basis. For each monomial basis element, a corresponding divergence-free elementary vector field is identified and integrated exactly (hence the solution is volume preserving). These elementary fields are composed by a splitting method, giving rise to a overall explicit volume-preserving method.

The method proposed are a significant contribution to the understanding of volume-preserving methods. To our knowledge, there is no explicit volume-preserving method that can deal with arbitrary polynomial vector fields of arbitrary degree, with exception of the case of linear and quadratic vector fields.

As far as stability is concerned, our methods were stable for sufficiently small step size, a feature that is common to explicit methods. 
An upper bound on the  step size can be obtained by Prop~\ref{th:dcase}. Beware that this bound can be unrealistic for long time numerical integration and further step size restriction might occur, especially if the solution is unbounded, which is often the case of divergence-free vector fields and and high degrees of nonlinearity. These problems are harder to integrate and most numerical methods will require step size reduction for stability.

To conclude, let us sketch out a possible way to attach the problem of volume-preserving integration for generic functions:
let $\{\mathbf{\phi}_i(\mathbf{x})\}$ a set of basis functions. Consider the decomposition
\begin{displaymath}
	p(\mathbf{x}) = \nabla \cdot \mathbf{f}(\mathbf{x}) = \sum_i p_i \mathbf{\phi}_i(\mathbf{x})
\end{displaymath}
in the given basis. Then, for each $i$, because of the independence of the basis functions, one must have $p_i=0$.
Thus, as long as we are able to recover all the terms in the given vector field contributing to the coefficient $p_i$, we will automatically have a volume-preserving split vector field. The difficulty is to find suitable basis functions, in the sense that the corresponding vector fields must be easy to integrate exactly. One such example is the monomial basis for polynomials, $\mathbf{x}^\mathbf{j}$, for which we have demonstrated the existence of exact formulas.

\section*{Acknowledgements}

The work of the first author has been supported by the Norwegian Research Council through the project ``Geometric Numerical Integration and Applications'', NFR project no.\ 191178/V30.

\bibliographystyle{plain}

\end{document}